\documentclass[11pt,twoside,reqno]{amsart}
\usepackage{times}
\usepackage{amsmath}
\usepackage{amssymb}
\usepackage{amsfonts}
\usepackage{amscd}
\usepackage{amsthm}
\usepackage{amsxtra}
\usepackage{stmaryrd}

\usepackage[all]{xy}  \CompileMatrices
\usepackage{mathrsfs}
\usepackage{nicefrac}
\usepackage{graphicx}
\usepackage{color}
\usepackage{wrapfig}

\newcommand{\todo}[1]{}

\theoremstyle{plain}
\newtheorem{theorem}{Theorem}

\newtheorem*{proposition*}{Proposition}

\newtheorem*{corollary*}{Corollary}
\newtheorem{lemma}[theorem]{Lemma}

\newtheorem*{theorem*}{Theorem}
\newtheorem*{lemma*}{Lemma}
\newtheorem*{conjecture*}{Conjecture}

\newtheorem*{question*}{Question}

\theoremstyle{definition}

\newtheorem*{exercise*}{Exercise}

\theoremstyle{remark}

\newtheorem*{remark*}{Remark}
\newtheorem{remsTh}[theorem]{Remarks}

\newcommand{\subclass}[1]{}

\newcommand{\enumTi}[1]{\renewcommand{\theenumi}{#1}}

\newcommand{\alphenumi}{\enumTi{\alph{enumi}}}

\newcommand{\romenumi}{\enumTi{\roman{enumi}}}

\renewcommand{\em}{\sl}

\newlength{\algotabbingwidth}
\theoremstyle{plain}

\newcommand{\cop}{c}

\begin{document}

\title[Cops \& Robber on non-orientable surfaces]{A note on the Cops \& Robber game on 
graphs embedded in non-orientable surfaces}%
\author[N.E.~Clarke]{Nancy E.~Clarke}
\address{Nancy E.~Clarke\\
  Department of Mathematics and Statistics\\
  Acadia University\\
  Wolfville, Nova Scotia, Canada}%
\email{nancy.clarke@acadiau.ca}
\author[S.~Fiorini]{Samuel Fiorini$^{1}$}
\address{Samuel Fiorini\\
  D\'epartement de Math\'ematique\\
  Universit\'e Libre de Bruxelles\\
  Brussels, Belgium}%
\email{sfiorini@ulb.ac.be}
\author[G.~Joret]{Gwena\"el Joret$^{1,2}$}
\address{Gwena\"el Joret\\
  D\'epartement d'Informatique\\
  Universit\'e Libre de Bruxelles\\
  Brussels, Belgium}%
\email{gjoret@ulb.ac.be}
\thanks{$^{1}$ This work was supported in part by the Actions de Recherche
Concert\'ees (ARC) fund of the Communaut\'e fran\c{c}aise de
Belgique.}
\thanks{$^{2}$ GJ is a Postdoctoral Researcher of the \textit{Fonds
National de la Recherche Scientifique (F.R.S.--FNRS)}}

\author[D.O.~Theis]{Dirk Oliver Theis$^{3}$}
\address{Dirk Oliver Theis\\
Fakult\"at f\"ur Mathematik \\
Otto-von-Guericke-Universit\"at Magdeburg\\
Magdeburg, Germany }%
\email{dirk.theis@ovgu.de}%
\thanks{$^{3}$ DOT supported by \textit{Fonds National de la Recherche Scientifique
(F.R.S.--FNRS)}}   




\begin{abstract}
  We consider the two-player, complete information game of Cops and Robber played on undirected, finite, reflexive 
graphs.  
A number of cops and one robber are positioned on
  vertices and take turns in sliding along edges.  The cops win if, after a move, a cop and the robber are on the same vertex.  The minimum number of cops
  needed to catch the robber on a graph is called the cop number of that graph.  

  Let $\cop(g)$ be the supremum over all cop numbers of graphs embeddable in a closed orientable surface of genus $g$, and likewise $\tilde \cop(g)$ for
  non-orientable surfaces.
  It is known (Andreae, 1986) that, for a fixed surface, the maximum over all cop numbers of graphs embeddable in 
this surface is finite.  More precisely,
  Quilliot (1985) showed that $\cop(g) \le 2g + 3$, and Schr\"oder (2001) sharpened this to $\cop(g)\le 
\frac32g + 3$.
  In his paper, Andreae gave the bound $\tilde \cop(g) \in O(g)$ with a weak constant, and posed the question whether 
a stronger bound can be obtained.
  Nowakowski \& Schr\"oder obtained $\tilde \cop(g) \le 2g+1$.

  In this short note, we show $\tilde \cop(g) \leq \cop(g-1)$, for any $g \ge 1$.
  As a corollary, using Schr\"oder's results, we obtain the following: the maximum cop number of graphs embeddable in 
the projective plane is 3; the maximum cop
  number of graphs embeddable in the Klein Bottle is at most 4, $\tilde \cop(3) \le 5$, and $\tilde \cop(g) \le \frac32g + 3/2$ for all other~$g$.
\end{abstract}
\maketitle

For an integer $k\ge 1$, the \textit {Cops and Robber game with $k$ cops} is a pursuit game 
played on a reflexive graph, i.e.~a graph with a
loop at every vertex.
There are two opposing sides, a set of $k$ cops and a single robber. The cops begin the game by each choosing
a (not necessarily distinct) vertex to occupy, and then the robber chooses a vertex.
The two sides move alternately, where a move is to slide along an edge or along a loop. The latter is equivalent to
passing were the game played on a loopless graph.
There is perfect
information, and the cops win if any of the cops and the
robber occupy the same vertex at the same time, after a finite number of moves.
Graphs on which
one cop
suffices to win are called {\it copwin} graphs. In general, we say that a graph $G$ is \textit{$k$-copwin} if $k$ cops can win on $G$. The minimum number of cops that suffice to win on 
$G$ is the cop number of $G$, denoted $c(G)$.
The game has been considered on infinite graphs but, here, we only consider finite graphs.

Nowakowski \& Winkler \cite{NowakowskiWinkler83} and Quilliot \cite{quil} have characterized the class of copwin graphs. The class of $k$-copwin graphs, $k > 1$, has been characterized by Clarke and MacGillivray \cite{meandgary}.
Families of graphs with unbounded cop number have been constructed \cite{AignerFromme84}, even families of $d$-regular graphs, for each $d\ge 3$ \cite{Andreae84}.

By a surface, we mean a closed surface, i.e.~a compact two dimensional topological manifold without boundary.
For any non-negative integer $g$, we denote by $\cop(g)$ the supremum over all $\cop(G)$, with $G$ ranging over all 
graphs embeddable in an orientable surface of
genus~$g$, and we call this the cop number of the surface.  Similarly, we define the cop number $\tilde \cop(g)$ of a non-orientable surface of genus $g$ to be
the supremum over all $\cop(G)$, with $G$ ranging over all graphs embeddable in this surface.

Aigner \& Fromme~\cite{AignerFromme84} proved that the cop number of the sphere is equal to three; i.e.~$\cop(0)=3$.  
Quilliot~\cite{Quilliot85} gave an
inductive argument to the effect that the cop number of an orientable surface of genus $g$ is at most $2g+3$.  Schr\"oder~\cite{Schroder01} was able to sharpen
this result to $\cop(g) \le \frac32 g +3$.  He also proved that the cop number of the double torus is at most 5.

Andreae~\cite{Andreae86} generalized the work of Aigner \& Fromme.  He proved that, for any graph $H$ satisfying a mild connectivity assumption, the class of
graphs which do not contain $H$ as a minor has cop number bounded by a constant depending on $H$.  Using this, 
and the well known formula for the non-orientable
genus of a complete graph, he obtained an upper bound for the cop number of a non-orientable surface of genus $g$, namely
\begin{equation*}
  \tilde\cop(g) \le \binom{ \lfloor  7/2 + \sqrt{6g + 1/4} \rfloor}{2}.
\end{equation*}

In an unpublished note, Nowakowski \& Schr\"oder \cite{NowakowskiSchroder09}, use a series of technically challenging arguments to prove a much stronger bound:
$\tilde\cop(g) \le 2g+1$.

In this short note, we prove the following.

\begin{theorem}\label{thm:copno-surface-genus}
 For a non-negative integer $g$, $\displaystyle \cop(\lfloor g/2 \rfloor) \le \tilde \cop(g) \le \cop(g-1).$
\end{theorem}

This immediately improves the best known upper bound for the non-orientable surface of genus $g$ to 
$\tilde\cop(g) \le \frac32(g-1) + 3 = \frac 32(g+1)$.
The following table gives the new and 
status quo for the concrete upper bounds.

\vspace{-5mm}

\begin{center}
  \bigskip%
  \begin{table}[h]
  \begin{tabular}{lllllllllllllllllll}
    \hline
    N/o genus  & 1 & 2 & 3               & 4 & 5  &  6 &  7 \\
    \hline
    N.~\&~S.~\cite{NowakowskiSchroder09} & 3 & 5 & 7               & 9 & 11 & 13 & 15 \\
    Here   & 3 & 4 & 5\footnotemark  & 7 & 9  & 10 & 12 \\
    \hline
  \end{tabular}
  \bigskip
  \footnotetext{Using Schr\"oder's \cite{Schroder01} result that $\cop(2)\le 5$.}
  \caption{Comparison of the new and status quo upper bounds for $\tilde\cop(g).$}
\label{tab1}
  \end{table}
\end{center}

\vspace{-1.1cm}

We say that a \textit{weak cover} of $H$ by $G$ is a surjective mapping $p\colon V(G) \to V(H)$ which maps 
vertex neighborhoods onto vertex neighborhoods; i.e.~for every vertex $u$ of $G$, we have $p(N(u)) = N(p(u))$.  (This 
terminology lends on the 
classical definition of a ``cover'' without weak, where the restriction to
the vertex neighborhood $p\colon N(u) \to N(p(u))$ is required to be a bijection.)
Using the same technique as for the inequality ``$\le$'' in the proof of Theorem~\ref{thm:copno-surface-genus}, it is possible to show the following:

\begin{lemma}\label{lem:copno-weak-cover}
  If $G$ is a weak cover of $H$, then $\cop(H) \le \cop(G)$.
\end{lemma}

This is similar in spirit to the seminal result of Berarducci \& Ingrigila \cite{BerarducciIntrigila93}, saying that if $H$ is a retract of $G$, then the same
inequality holds.  Note, however, that neither of the two notions generalizes the other.  We will not prove Lemma~\ref{lem:copno-weak-cover}; the proof is only
slightly more technical than the geometric proof of Theorem~\ref{thm:copno-surface-genus}.

\section*{Proof}

Familiarity with the classification of combinatorial surfaces is assumed.  See any standard textbook on topology, 
such as~\cite{cain}.  We will make use of the standard representation of
surfaces as quotients of polygonal discs with labelled and directed edges.  Each label occurs twice, and the two edges with the same label are identified
according to their orientations.  Reading the labels of the edges in counterclockwise (i.e.~\textit{positive}) order 
and adding an exponent $-1$ whenever the
orientation of the edges is negative (i.e.~clockwise) gives the \textit{word} of the surface.

For a graph $G$, let $\gamma(G)$ denote the smallest integer $g$ such that $G$ can be embedded in an
orientable surface of genus $g$; similarly define $\tilde\gamma(G)$ as the smallest integer $g$ such that $G$ can be embedded in an non-orientable surface of
genus $g$. For the proof of Theorem 1, we use the following well-known fact. Its proof can be found in~\cite{mohar}.

\begin{lemma}[Folklore]\label{lem:nor-to-or-genuschange}
  For any graph $G$, $\tilde\gamma(G) \le 2\gamma(G) +1$.
\end{lemma}

In the proof of the inequality $\cop(g) \le \cop(g-1)$, we make use of the 
well-known fact that 
every manifold $X$ has a 2-sheeted covering $X'\to X$ by an orientable manifold.  If $X$ is a
non-orientable surface of genus $g$, it is easy to see that the standard construction 
(again, see~\cite{cain}) yields a surface of genus $g-1$. This is 
Lemma~\ref{lem:covering-or-nonor}. The proof is straightforward (consider Figure 1), and is thus omitted.

\begin{lemma}\label{lem:covering-or-nonor}
  A non-orientable surface of genus~$g$ has an orientable surface of genus $g-1$ as a 2-sheeted covering space.
\end{lemma}

\begin{figure}[htbp]
\scalebox{0.75}{\input{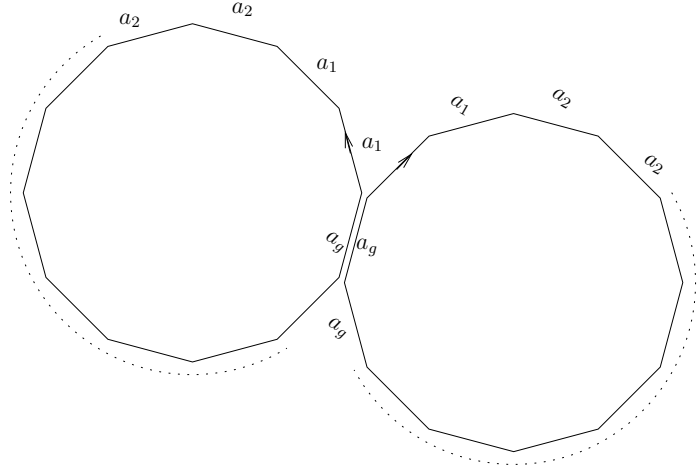}}
\caption{A figure to accompany Lemma 4.}
\label{fig1}
\end{figure}

\medskip

We are now ready for the proof of our main result.

\bigskip

\noindent \textbf{Proof of Theorem 1.}
Lemma 3 immediately implies that $\cop(g) \le \tilde \cop(2g+1)$, and hence $\tilde \cop(g) \ge 
\cop(\lfloor g/2 \rfloor)$.

For the proof of the remaining inequality $\tilde \cop(g) \le \cop(g-1)$, let $X$ be the non-orientable surface of 
genus~$g$ on which a graph $G$ is embedded.  We identify the graph
$G$ with its embedding; i.e.~we think of the vertex set $V(G)$ as a set of points of $X$ and the edge set of $E(G)$ 
as a set of internally disjoint injective
curves connecting the respective end vertices of the edge.
  
By Lemma~\ref{lem:covering-or-nonor}, there exists a covering $p\colon X'\to X$ of $X$ by an orientable surface $X'$ with genus $g' := g-1$.  Consider the graph
$G'$ whose vertex set is $\{p^{-1}(V(G))\}$ and whose edge set consists of the curves obtained by lifting the edges of $G$.  By construction, $G'$ is embedded
in the orientable surface $X'$ of genus~$g'$.

We now give a strategy for $k := \cop(g')$ cops to win the Cops and Robber game on $G$, by ``simulating'' a game on $G'$ and using any winning strategy for $k$ cops
on this graph, who chase an ``imaginary'' robber.  In such a strategy, the $k$ cops first choose their starting 
vertices $u_1,\dots,u_k \in V(G')$.  In the
strategy for $G$, we let the starting vertices be $p(u_1),\dots,p(u_k)$.
Suppose now that, in the game on $G$, the robber chooses a starting vertex $r$.  We choose an arbitrary starting vertex for an 
imaginary robber on $G'$ arbitrarily in the
fibre $p^{-1}(r)$.

Throughout the game, the position of each player in $G'$ will be in the fibre $p^{-1}(x)$ of the position $x$ of the 
corresponding player in $G$.  Moreover,
the movements of the players on $G$ describe curves on $X$, which can be lifted (uniquely, although this is 
not essential) 
to curves on $X'$ forming
walks in $G'$.

Now, whenever it be the cops' turn in any game on $G$, the robber is at a certain vertex $s$ of $G'$, and the $k$ cops are on vertices $v_1,\dots,v_k$.  The
strategy for the cops on $G'$ now prescribes moves for the cops.  The corresponding moves in $G$ are then given as images under $p$.

Since we have a winning strategy, after a finite number of moves, the ``imaginary robber" on $G'$ will be on the same 
vertex as a cop in $G'$.  Consequently, the
same holds on $G$, and thus the cops have won the game on $G$.
\hfill$\Box$

\section*{Conclusion}

We conclude with a conjecture.

\begin{conjecture*}
  For a non-negative integer $g$, $\displaystyle \tilde\cop(g) = \cop(\lfloor g/2 \rfloor)$.
\end{conjecture*}

One might wonder whether it is possible to improve Theorem~\ref{thm:copno-surface-genus} by taking a different covering, or possibly a branched covering.  This
is impossible: It is a well-known fact that, whenever $p\colon X'\to X$ is a (branched) covering with $X'$ orientable and $X$ non-orientable, then $p$ lifts to
a (branched) covering $\tilde p\colon X'\to\tilde X$, where $\tilde X$ is the orientable double cover constructed in Lemma~\ref{lem:covering-or-nonor}.



\begin{thebibliography}{1}

\bibitem{AignerFromme84}
M.~Aigner and M.~Fromme.
\newblock A game of cops and robbers.
\newblock {\em Discrete Appl.\ Math.}, 8:1--12, 1984.

\bibitem{Andreae84}
T.~Andreae.
\newblock Note on a pursuit game played on graphs.
\newblock {\em Discrete Appl.\ Math.}, 9:111--115, 1984.

\bibitem{Andreae86}
T.~Andreae.
\newblock On a pursuit game played on graphs for which a minor is excluded.
\newblock {\em J.\ Combin.\ Th., Ser.\ B}, 41:37--47, 1986.

\bibitem{BerarducciIntrigila93}
A.~Berarducci and B.~Intrigila.
\newblock On the cop number of a graph.
\newblock {\em Adv.\ Appl.\ Math.}, 14:389--403, 1993.

\bibitem{cain}
G.L.~Cain.
\newblock {\em Introduction to General Topology}.
\newblock Addison-Wesley, 1994.

\bibitem{meandgary} N.E.~Clarke and G.~MacGillivray. 
\newblock Characterizations of $k$-copwin graphs. 
\newblock Submitted to \emph{Discrete Math.}

\bibitem{mohar} B.~Mohar and C.~Thomassen.
\newblock {\em Graphs on Surfaces}.
\newblock Johns Hopkins University Press, 2001.

\bibitem{NowakowskiSchroder09}
R.~{J.} Nowakowski and B.~{S. W.} Schr{\"o}der.
\newblock Bounding the cop number using the crosscap number.
\newblock Unpublished note, 1997.

\bibitem{NowakowskiWinkler83}
R.~{J.} Nowakowski and P.~Winkler.
\newblock Vertex to vertex pursuit in a graph.
\newblock {\em Discrete Math.}, 43:23--29, 1983.

\bibitem{quil} A.~Quilliot.
\newblock \emph{Jeux et Points Fixes sur les graphes}.
\newblock Th\`{e}se de 3\`eme cycle, Universit\'{e} de Paris VI, 1978, 131--145. 

\bibitem{QuilliotTh} A.~Quilliot, 
\newblock \emph{Probl\`emes de jeux, de point fixe, de connectivit\'e et de repr\'esentation sur 
des graphes, des ensembles ordonn\'es et des hypergraphes}. 
\newblock Th\`{e}se d'Etat, Universit\'{e} de Paris VI, 1983.

\bibitem{Quilliot85}
A.~Quilliot.
\newblock A short note about pursuit games played on a graph with a given
  genus.
\newblock {\em J.\ Combin.\ Theory Ser.\ B}, 38:89--92, 1985.







\bibitem{Schroder01}
B.~{S. W.} Schr{\"o}der.
\newblock The copnumber of a graph is bounded by
  {$\lfloor\frac32\mathrm{genus}(G)\rfloor+3$}.
\newblock In {\em ``Categorical Perspectives'' --- Proceedings of the
  Conference in Honor of George Strecker's 60th Birthday}, pages 243--263.
  Birkh{\"a}user, 2001.

\end{thebibliography}
\end{document}